\documentclass{article}

\usepackage[T2A]{fontenc}
\usepackage[utf8]{inputenc}
\usepackage[english,russian]{babel}
\usepackage[tbtags]{amsmath}
\usepackage{amsfonts,amssymb}
\usepackage{algorithm}
\usepackage{algorithmic}
\usepackage{mathtools}
\usepackage{dirtytalk}
\usepackage{tikz} 

\usepackage{mathrsfs}

\usepackage[paper]{mz2ewinUTF}

\usepackage{graphicx}
\usepackage{url}

\newcommand*{\e}{\varepsilon}
\newcommand*{\la}{\langle}
\newcommand*{\ra}{\rangle}

\def\E{\mathbb E}

\def\ag#1{{\color{black}#1}}

\def\at#1{{\color{black}#1}}
\def\att#1{{\color{black}#1}}
\def\attt#1{{\color{black}#1}}
\def\atttt#1{{\color{black}#1}}

\numberwithin{equation}{section}

\theoremstyle{plain}

\usepackage{pifont}

\newtheorem{theorem}{Теорема}
\newtheorem{assumption}{Предположение}
\newtheorem{lemma}{Лемма}
\newtheorem{proof}{Доказательство}

\newcommand*{\norm}[1]{\left\lVert#1\right\rVert_2}

\begin{document}

\udk{519.85}


\author{Д.М.~Двинских}
\address{Weierstrass Institute, Berlin; \\ Московский физико-технический институт, Москва; \\ Институт проблем передачи информации РАН}
\email{darina.dvinskikh@wias-berlin.de}

\author{А.И.~Тюрин}
\address{Высшая школа экономики, Москва}
\email{alexandertiurin@gmail.com}


\author{А.В.~Гасников}
\address{Московский физико-технический институт, Москва; \\
Институт проблем передачи информации РАН; \\ Высшая школа экономики}
\email{gasnikov@yandex.ru}

\author{C.C.~Омельченко}
\address{Московский физико-технический институт, Москва}
\email{sergey.omelchenko@phystech.edu}

\title{Ускоренный и неускореный стохастический градиентный спуск в модельной общности}

\markboth{Д.\,М.~Двинских, А.\,И.~Тюрин,
А.\,В.~Гасников, C\,.C.~Омельченко}
{Ускоренный и неускореный стохастический градиентный спуск}

\maketitle

\begin{fulltext}
\begin{abstract}
В статье описывается новый способ получения оценок скорости сходимости оптимальных методов решения задач гладкой (сильно) выпуклой стохастической оптимизации. Способ базируется на получение результатов стохастической оптимизации на основе результатов о сходимости оптимальных методов в условиях неточных градиентов с малыми шумами неслучайной природы. В отличие от известных ранее результатов в данной работе все оценки получаются в модельной общности. 

 Библиография: 12 названий.

\end{abstract}


\footnotetext{Работа А.И. Тюрина в п. 3 поддержана грантом РФФИ 19-31-90062 Аспиранты. Работа А.В. Гасникова в п. 2 была поддержана грантом РФФИ 18-31-20005 мол\_а\_вед.
Работа Д.М.~Двинских в п. 1 была выполнена при поддержке Министерства науки и высшего образования Российской Федерации (госзадание) № 075-00337-20-03, номер проекта 0714-2020-0005.}

 \section{Введение} \label{section_1}
В данной работе рассматривается задача стохастической оптимизации \cite{gasnikov2017,devolder2013,lan2019}
\begin{equation}
\label{Problem}
f(x) = \E[f(x,\xi)] \to \min_{x\in Q \subseteq \mathbb{R}^n}, 
\end{equation}
где множество $Q$ предполагается выпуклым и замкнутым, \at{$\xi$~--- случайная величина, математическое ожидание $\E[f(x,\xi)]$ определено и конечно для любого $x \in Q$,}  функция $f(x)$ -- $\mu$-сильно выпуклая в 2-норме ($\mu \geq 0$) и имеющая $L$-Липшицев градиент, т.е. для всех $x,y\in Q$
\at{$$f(x) + \la \nabla f(x), y - x \ra + \frac{\mu}{2}\|y-x\|_2^2 \leq f(y) \leq f(x) + \la \nabla f(x), y - x \ra + \frac{L}{2}\|y-x\|_2^2.$$}
   
Предположим, что есть доступ к $\nabla f(x,\xi)$ -- стохастическому градиенту $f(x)$, удовлетворяющему следующим условиям\footnote{Заметим, что для задач минимизации функционалов вида суммы условие ограниченности (субгауссовской) дисперсии может не выполняться даже в очень простых (квадратичных) ситуациях. Как следствие, в общем случае приводимые далее результаты не распространяются на задачи минимизации функционалов вида суммы, в которых в качестве стохастического градиента выбирается градиент случайно выбранного слагаемого \cite{assran2020}.} (несмещенность и субгауссовость хвостов распределения, с субгауссовской дисперсией $\sigma^2$)
\begin{equation}
\label{stoch_gradient_subgaussion}
\E \left[\nabla f(x, \xi)\right] \equiv \nabla f(x),
\E\left[\exp\left(\frac{ \|\nabla f(x, \xi)-\E[\nabla f(x, \xi)]\|_2^2}{\sigma^2}\right)\right]\le \exp(1),
\end{equation}
для всех $x \in Q$.

Тогда после $N$ вычислений $\nabla f(x,\xi)$ с большой вероятностью имеем\footnote{Здесь и далее \textquotedblleftс большой вероятностью\textquotedblright\ -- означает с вероятностью $\ge 1 - \gamma$, а $\tilde{O}(\cdot)$ означает то же самое, что $O(\cdot)$, только числовой множитель зависит от $\ln \left( N/\gamma\right)$.} \cite{gasnikov2017,devolder2013,lan2019}
\begin{equation}
\label{estimate_stochastic}
f(x_N) - f(x_*) = \tilde{O}\left(\min\left\{\frac{LR^2}{N^p} + \frac{\sigma R}{\sqrt{N}},
\at{LR^2}
\exp\left(-\left(\frac{\mu}{L}\right)^{\frac{1}{p}}\frac{N}{2}\right) + \frac{\sigma^2}{\mu N}\right\}\right),
\end{equation}
где $x_*$ -- решение задачи \eqref{Problem}, $R = \|x_0 - x_*\|_2$, $x_0$ -- точка старта,
$p=1$ отвечает стохастическому градиентному спуску,
а $p=2$ ускоренному стохастическому спуску. 

С другой стороны известно (см. \cite{gasnikov2017,devolder2013,turin2019,gorbunov2019,stonyakin2019}), что если для задачи \eqref{Problem} доступен неточный градиент $\nabla_{\delta} f(x)$, удовлетворяющий для всех $x,y \in Q$ ослабленному условию $L$-Липшицевости градиента
\begin{equation}
	f(x) + \la \nabla_{\delta} f(x), y-x \ra + \frac{\mu}{2}\|y-x\|^2_2  - \delta_1 \le f(y)  \le f(x) + \la \nabla_{\delta} f(x), y-x \ra + \frac{L}{2}\|y-x\|^2_2 + \delta_2,
\end{equation}
то после $N$ вычислений $\nabla_{\delta} f(x)$ для соответствующих модификаций градиентного и ускоренного градиентного спуска можно получить оценку, аналогичную  \eqref{estimate_stochastic}\footnote{\att{Для $p=2$ нужно ввести дополнительные ограничения на $\delta_1$, см. ниже}.}
\begin{equation}
\label{estimate_inexact}
\tilde{O}\left(\min\left\{\frac{LR^2}{N^p} + \delta_1 + N^{p-1}\delta_2,
\at{LR^2}
\exp\left(-\left(\frac{\mu}{L}\right)^{\frac{1}{p}}\frac{N}{2}\right) + \delta_1 + \left(\frac{L}{\mu}\right)^{\frac{p-1}{2}}\delta_2 \right\}\right).
\end{equation}


В данной статье подмечается, что результат \eqref{estimate_stochastic} может быть получен\footnote{\ag{В сильно выпуклом случае только в смысле сходимости по математическому ожиданию, без оценки вероятностей больших отклонений.}} из результата \eqref{estimate_inexact}. 
Более того, сделанное наблюдение, оказывается возможным провести и в модельной общности.

Данная работа имеет следующую структуру. В разделе~\ref{section_2} рассматривается концепция неточного градиента функции и для нее приводится соответствующая теорема сходимости, дополнительно для задачи из \eqref{Problem} приводится простой способ того, как можно получить оптимальные оценки. В разделе~\ref{section_3} рассматривается концепция неточной модели функции и доказываются все основные результаты. 
\atttt{Отметим, что в начале раздела~\ref{section_3} и разделе~\ref{subsection:maxmin} приводятся примеры некоторых классов негладких задач (композитная оптимизация и оптимизация максимума нескольких гладких функций), для которых возможно применять предложенные концепции неточной модели функции. Это позволяет говорить о том, что полученные в работе результаты об оценках сложности, эффективные на классах выпуклых гладких задач, верны и для некоторых типов выпуклых негладких задач.}

\section{Основные результаты} \label{section_2}
Ограничимся для компактности изложения пояснением перехода от \eqref{estimate_inexact} к \eqref{estimate_stochastic} для случая $\mu = 0$, и с теми же целями переопределим $R = \max_{x,y \in Q} \|x-y\|_2$ (в действительности, все приведенные далее в этом разделе результаты верны для $R = \|x_0 - x_*\|_2$; показывается аналогично \cite{gorbunov2019}). \att{Будем далее дополнительно предполагать, что в ослабленном условии $L$-Липшицевости градиента для неточного градиента $\nabla_{\delta} f(x)$ ошибка $\delta_1$ зависит от $y$ и $x$:
\begin{equation}
\label{inexact_oracle_delta_1}
	f(x) + \la \nabla_{\delta} f(x), y-x \ra  - \delta_1(y,x) \le f(y)  \le f(x) + \la \nabla_{\delta} f(x), y-x \ra + \frac{L}{2}\|y-x\|^2_2 + \delta_2,
\end{equation}
}
Первое важное наблюдение заключается, в следующем (доказательство более общего утверждения вынесено в Раздел 3).
\at{
\begin{assumption}
\label{assumption:delta}
Пусть даны две последовательности \att{$\delta_1^k(y,x)$} и $\delta_2^k$ ($k \geq 0$). Будем предполагать, что
\begin{center}
$\E \left[\delta_1^k(y,x) |\delta_{1,2}^{k-1},\delta_{1,2}^{k-2},... \right]= 0$,  (условная несмещенность)
\end{center}
 $\delta_1^k(y,x)$ имеет  $\left(\hat{\delta}_1\right)^2$-субгауссовскую условную дисперсию, $\sqrt{\delta_2^k}$ имеет $\hat{\delta}_2$-суб\-гаус\-совс\-кий условный  второй момент.
\end{assumption}
}

\att{
\begin{assumption}
\label{assumption:delta_4}
Пусть даны две последовательности $\delta_1^k(x,y)$ и $\delta_2^k$ ($k \geq 0$). Случайная величина $\delta_1^k(x, y)$ имеет  $\left(\hat{\delta}_1^k(x - y)\right)^2$-субгауссовский условный момент ($\hat{\delta}_1^k(\cdot)$ есть неслучайная функция от одного аргумента) такой, что 
\begin{enumerate}
    \item $\hat{\delta}_1^k(\alpha z) \leq \alpha \hat{\delta}_1^k(z)$ для всех $\alpha \geq 0$ и $z \in B(0, R)$.
    \item $\hat{\delta}_1 < +\infty$, где $\hat{\delta}_1 \geq \sup_{z \in B(0, R)} \hat{\delta}_1^k(z)$.
\end{enumerate}
\end{assumption}
}

\att{
\begin{theorem}
\label{thm1}
    Пусть для последовательностей $\delta_1^k(y,x)$ и $\delta_2^k$ ($k \geq 0$) верно предположение~\ref{assumption:delta}, тогда 
    \begin{enumerate}
    \item 
    После $N$ шагов соответствующей модификации градиентного спуска будет верно следующее неравенство:
    \begin{align*}
    \E[f(x_N)] - f(x_*) \leq O\left(\frac{LR^2}{N} + \hat{\delta}_2\right).
    \end{align*}
    И с большой вероятностью
    \begin{equation*}
    f(x_N) - f(x_*) = \tilde{O}\left(\frac{LR^2}{N} + \frac{\hat{\delta}_1}{\sqrt{N}} + \hat{\delta}_2\right).
    \end{equation*}
    \item
    После $N$ шагов соответствующей модификации ускоренного градиентного спуска будет верно следующее неравенство:
    \begin{align*}
    \E[f(x_N)] - f(x_*) \leq O\left(\frac{LR^2}{N^2} + N \hat{\delta}_2\right).
    \end{align*}
    Если дополнительно верно предположение~\ref{assumption:delta_4}, то с большой вероятностью
    \begin{equation*}
    f(x_N) - f(x_*) = \tilde{O}\left(\frac{LR^2}{N^2} + \frac{\hat{\delta}_1}{\sqrt{N}} + N\hat{\delta}_2\right).
    \end{equation*}
    \end{enumerate}
\end{theorem}
}

Данная теорема является следствием теоремы~\ref{theorem:gm_model_subgaussion} и \ref{theorem:fgm_model_subgaussion} из Раздела 3 для модели вида $\psi_{\delta}(y,x) = \la \nabla_{\delta} f(x), y-x \ra$.

\ag{Если 
в качестве $\nabla_{\delta}f(x)$ взять $\nabla f(x,\xi)$, тогда будет выполнено неравенство \eqref{inexact_oracle_delta_1} с
$\delta_1(y,x) = \la \nabla f(x) - \nabla f(x,\xi), y - x \ra$, $\delta_2 = \frac{1}{2L}\|\nabla f(x,\xi) - \nabla f(x)\|_2^2$,
с $L:=2L$. Чтобы это понять, достаточно заметить (первое важное наблюдение), что $$\la \nabla f(x) - \nabla f(x,\xi) , y - x \ra \le \frac{1}{2L}\|\nabla f(x,\xi) - \nabla f(x)\|_2^2 + \frac{L}{2}\|y-x\|_2^2.$$} Более того, для $\delta_1(y,x)$ и $\delta_2$ верно (см. обозначения предположения~\ref{assumption:delta}), что $\hat{\delta}_1 = O(\sigma R)$ и $\hat{\delta}_2 = O(\sigma^2/L)$, \att{и для $\delta_1(y,x)$ верно предположение~\ref{assumption:delta_4}}.

Сделанное наблюдение позволяет \ag{с помощью теоремы~\ref{thm1} } получить, например, что с большой вероятностью
\begin{equation}\label{estimate_new1}
f(x_N) - f(x_*) = \tilde{O}\left(\frac{LR^2}{N^p} + \frac{\sigma R}{\sqrt{N}} + N^{p-1}\frac{\sigma^2}{L}\right),
\end{equation}
причем
\begin{equation*}
\E[f(x_N)] - f(x_*) = O\left(\frac{LR^2}{N^p} +  N^{p-1}\frac{\sigma^2}{L}\right).
\end{equation*}

\ag{Отметим также возможность при $p=1$ выбора шага $h$ в базовом детерминированном градиентном спуске меньше чем $1/L$. 
В этом случае оценка будет иметь вид $$\E[f(x_N)] - f(x_*) = O\left(\frac{R^2}{hN} +  h\sigma^2\right).$$
Минимизируя правую часть по $h$, получим $h = R/\left(\sigma\sqrt{N}\right)$ и $$\E[f(x_N)] - f(x_*) = O\left(\frac{\sigma R}{\sqrt{N}}\right).$$
Аналогичные оценки можно выписать и в категориях больших отклонений. 
}

Вторым важным наблюдением является следующая теорема (см., например, \cite{gorbunov2019}). 
\begin{theorem}\label{thm2}(Батчинг)
 Пусть $\{\xi^l\}_{l=1}^r$ -- независимые одинаково распределенные случайные величины (также как случайная величина $\xi$, которая имеет субгауссовскую дисперсию $\sigma^2$). Тогда для $\sigma^2_r$ -- субгауссовской дисперсии $$\nabla^r f\left(x,\{\xi\}_{l=1}^r\right)=\frac{1}{r}\sum_{l=1}^r \nabla f(x,\xi^l),$$ справедлива оценка $\sigma^2_r = O\left(\sigma^2/r\right)$. 
\end{theorem}

Для обоснования перехода от \eqref{estimate_inexact} к \eqref{estimate_stochastic} положим в \attt{\eqref{inexact_oracle_delta_1}}  $$\nabla_{\delta} f(x) = \nabla^r f\left(x,\{\xi\}_{l=1}^r\right)$$ и подберем должным образом $r$. Для подбора $r$ потребуем, чтобы правая часть в оценке \eqref{estimate_stochastic} была равна $\e$ (желаемой точности решения задачи по функции). Чтобы добиться этого исходя из формулы \eqref{estimate_new1} согласно теореме~\ref{thm2} нужно выбрать $r$ так, чтобы все слагаемые в \eqref{estimate_new1} были порядка $\e$. То есть 
\begin{center}
$\left(\frac{LR^2}{N^p}\right) \simeq \e$,  $\frac{\sigma R}{\sqrt{N}} \simeq \e$, $N^{p-1}\frac{\sigma^2}{L}\simeq \e$.     
\end{center}
Получается переопределенная система уравнений на $N,r$, которая, тем не менее, оказывается совместной 
\att{$r = \tilde{O}\left(\frac{\sigma^2}{L\e}\left(\frac{LR^2}{\e}\right)^\frac{p-1}{p}\right)$}. При этом, 
\begin{center}
число итераций алгоритма -- 
$N = \tilde{O}\left(\left(\frac{LR^2}{\e}\right)^{\frac{1}{p}}\right)$,
\end{center}
\begin{center}
а число вычислений $\nabla f(x,\xi)$ -- $\tilde{O}\left(\max\left\{\left(\frac{LR^2}{\e}\right)^{\frac{1}{p}},\frac{\sigma^2R^2}{\e^2}\right\}\right) = \tilde{O}\left(\left(\frac{LR^2}{\e}\right)^{\frac{1}{p}} + \frac{\sigma^2R^2}{\e^2}\right)$.
\end{center}
Данные оценки в точности соответствуют тому, что можно получить с помощью батчинга из оценки \eqref{estimate_stochastic}. Отметим, что при  $p=2$ данные оценки оптимальны как по числу итераций, так и по числу параллельно вычисляемых стохастических градиентов на каждой итерации \cite{woodworth2018}.

\section{Модельная общность} \label{section_3}
Результаты раздела~\ref{section_2} можно воспроизвести и в модельной общности \cite{gasnikov2017,turin2019,stonyakin2019}. Будем говорить, что функция $\psi_{\delta}(y,x)$ является $(\delta,L)$-моделью целевой функции $f(x)$, если для всех $x,y\in Q$ функция $\psi_{\delta}(y,x)$ -- выпукла по $y$, $\psi_{\delta}(x,x)\equiv 0$,
\begin{equation}\label{inexact_model}
f(x) + \psi_{\delta}(y, x) + \frac{\mu}{2}\|y-x\|^2_2 - \delta_1 \le f(y)  \le f(x) + \psi_{\delta}(y, x)+ \frac{L}{2}\|y-x\|^2_2 + \delta_2.
\end{equation}

Для задач композитной оптимизации (см., например, \cite{gasnikov2017,nesterov2018}), в которых целевая функция имеет вид $F(x) = f(x)+h(x)$, где $h(x)$ достаточна простая функция, для которой доступен субградиент, а функция $f(x)$ имеет $L$-Липшицев градиент, и для нее доступен только стохастический градиент $\nabla f(x,\xi)$, в качестве модели можно взять $\psi_{\delta} (y,x) = \la \nabla^r f\left(x,\{\xi\}_{l=1}^r\right), y - x \ra + h(y) - h(x)$. Тогда аналогично разделу~\ref{section_2}, получим, что в  \eqref{inexact_model} можно положить $L:=2L$, $\hat{\delta}_1 = O(\sigma R/r)$, $\hat{\delta}_2 = O(\sigma^2/(Lr))$. Это наблюдение позволяет перенести все результаты раздела~\ref{section_2} на задачи стохастической композитной оптимизации.

Введем следующее предположение.
\begin{assumption}
\label{assumption:delta_2}
Пусть даны две последовательности $\delta_1^k$ и $\delta_2^k$ ($k \geq 0$). Будем предполагать, что имеется некоторая константа $\tilde{\delta}_1$ такая, что
\begin{center}
$\E \left[\delta_1^k |\delta_{1,2}^{k-1},\delta_{1,2}^{k-2},... \right] \leq \tilde{\delta}_1$,
\end{center}
$\delta_1^k$ имеет  $\left(\hat{\delta}_1\right)^2$-субгауссовский условный второй момент, $\sqrt{\delta_2^k}$ имеет $\hat{\delta}_2$-суб\-гаус\-совс\-кий условный  второй момент.
\end{assumption}
Отметим, что предположение~\ref{assumption:delta_2} является более общим, чем предположение~\ref{assumption:delta}.

Обозначим $q = 1 - \frac{\mu}{L}$. Представим градиентный и быстрый градинетный метод в модельной общности (Алгоритм 1 и 2). В разделах \ref{section:gm} и \ref{section:fgm} представлены теоремы сходимости и соотвествующие доказательства.

\begin{algorithm}
\caption{Градиентный метод}
\label{algorithm:GM}
\begin{algorithmic}[1]
\STATE \textbf{Input:} Начальная точка $x_0$, константа сильной выпуклости $\mu \geq 0$, константа липшевости градиента $L > 0$.
\FOR{$k \geq 0$}
\STATE
\begin{equation*}
\phi_{k+1}(x) := \psi_{\delta_k}(x, x_k)+ \frac{L}{2} \|x-x_k\|^2_2,
\end{equation*}
\begin{equation}
\label{gm_step}
x_{k+1} := \arg\min_{x \in Q} \phi_{k+1}(x).
\end{equation}
\ENDFOR
\STATE \textbf{Output:} $y_N = \frac{1}{\sum_{i = 1}^{N} q^{N-i}}\sum_{i = 1}^{N} q^{N-i} x_i$ и 
\end{algorithmic}
\end{algorithm}


\begin{algorithm}
\caption{Быстрый градиентный метод}
\label{algorithm:FGM}
\begin{algorithmic}[1]
\STATE \textbf{Input:} Начальная точка $x_0$,
константа сильноый выпуклости $\mu \geq 0$, константа липшевости градиента $L > 0$.
\STATE Set
$y_0 := x_0$, $u_0 := x_0$, $\alpha_0 := 0$, $A_0 := \alpha_0$
\FOR{$k \geq 0$}
\STATE Константа $\alpha_{k+1}$ --- это наибольший корень уравнения  
\begin{gather}
\label{alpha_def_strong}
A_{k+1}{(1 + A_k \mu)}=L\alpha^2_{k+1},\quad A_{k+1} := A_k + \alpha_{k+1}.
\end{gather}
\STATE \begin{equation}y_{k+1} := \frac{\alpha_{k+1}u_k + A_k x_k}{A_{k+1}}.\label{y_def}\end{equation}
\STATE
\begin{equation*}
\phi_{k+1}(x)=\alpha_{k+1}\psi_{\delta_k}(x, y_{k+1}) + \frac{1 + A_k\mu}{2} \|x-u_k\|^2_2 + \frac{\alpha_{k+1} \mu}{2} \|x-y_{k+1}\|^2_2,
\end{equation*}
\begin{equation}\label{fgm_step}
u_{k+1} := \arg\min_{x \in Q}\phi_{k+1}(x).
\end{equation}
\STATE
\begin{gather}
x_{k+1} := \frac{\alpha_{k+1}u_{k+1} + A_k x_k}{A_{k+1}}. \label{x_def}
\end{gather}
\ENDFOR
\STATE \textbf{Output:} $x_N$,
\end{algorithmic}
\end{algorithm}

\subsection{Градиентный метод для оптимизационных
задач, допускающих модель функции}
\label{section:gm}

\begin{lemma}
\label{lemma:str_conv}
	Пусть $\psi(x)$~--- выпуклая функция и
	\begin{gather*}
	y = {\arg\min_{x \in Q}} \{\psi(x) + \tfrac{\beta}{2}\norm{x - z}^2 + \tfrac{\gamma}{2}\norm{x - u}^2\},
	\end{gather*}
	где $\beta \geq 0$ и $\gamma \geq 0$.
Тогда
	\begin{align*}
\psi(x) &+ \tfrac{\beta}{2}\norm{x - z}^2 + \tfrac{\gamma}{2}\norm{x - u}^2\\
&\geq \psi(y) + \tfrac{\beta}{2}\norm{y - z}^2 + \tfrac{\gamma}{2}\norm{y - u}^2 + \tfrac{\beta + \gamma}{2}\norm{x - y}^2 ,\,\,\, \forall x \in Q.
	\end{align*}
\end{lemma}

\begin{proof}
    Из критерия оптимальности следует, что:
    \begin{gather*}
		\exists g \in \partial\psi(y), \,\,\, \langle g + \tfrac{\beta}{2} \nabla_y \norm{y - z}^2 + \tfrac{\gamma}{2} \nabla_y \norm{y - u}^2, x - y \rangle \geq 0 ,\,\,\, \forall x \in Q.
	\end{gather*}
    Из $\beta + \gamma$--сильной выпуклости $\psi(x) + \tfrac{\beta}{2}\norm{x - z}^2 + \tfrac{\gamma}{2}\norm{x - u}^2$ получаем, что
    \begin{align*}
		\psi(x) + \tfrac{\beta}{2}\norm{x - z}^2 &+ \tfrac{\gamma}{2}\norm{x - u}^2 \geq \psi(y) + \tfrac{\beta}{2}\norm{y - z}^2 + \tfrac{\gamma}{2}\norm{y - u}^2 \\
		&+ \langle g + \tfrac{\beta}{2} \nabla_y \norm{y - u}^2 + \tfrac{\gamma}{2} \nabla_y \norm{y - z}^2, x - y\rangle + \tfrac{\beta + \gamma}{2}\norm{x - y}^2
	\end{align*}
	Последние два неравенства доказывают лемму.
\end{proof}

\begin{theorem}
\label{theorem:gm_model}
    После $N$ шагов Алгоритма \ref{algorithm:GM} будет верно следующее неравенство:
    \begin{align*}
    f(y_{N}) - f(x_*) \leq \min\left\{\frac{LR^2}{2N}, \frac{LR^2}{2}\exp\left(-\frac{\mu}{L}N\right)\right\} + \frac{1}{\sum_{i=1}^{N}q^{N-i}}\sum_{i=1}^{N}q^{N-i}(\delta_1^{i-1} + \delta_2^{i-1}).
    \end{align*}
\end{theorem}

\begin{proof}
Из \atttt{\eqref{inexact_model}} получаем:
\begin{align*}
f(x_{N}) 
&\leq f(x_{N-1}) + \psi_{\delta_{N-1}}(x_{N}, x_{N-1}) + \frac{L}{2}\norm{x_{N} - x_{N-1}}^2 + \delta_2^{N-1}.
\end{align*}
Воспользуемся леммой~\ref{lemma:str_conv} для \eqref{gm_step}:
\begin{align*}
f(x_{N})
&\leq f(x_{N-1}) + \psi_{\delta_{N-1}}(x, x_{N-1}) + \frac{L}{2}\norm{x - x_{N-1}}^2 - \frac{L}{2}\norm{x - x_{N}}^2 + \delta_2^{N-1}.
\end{align*}
Воспользуемся левым неравенством из \eqref{inexact_model}:
\begin{align}
\label{analysis_algorithm_grad:prove_1}
f(x_{N}) \leq f(x) + \frac{L - \mu}{2}\norm{x - x_{N-1}}^2 - \frac{L}{2}\norm{x - x_{N}}^2 + \delta_1^{N-1} + \delta_2^{N-1}.
\end{align}
Перепишем неравенство для $x = x_*$:
\begin{align*}
\frac{1}{2}\norm{x_* - x_{N}}^2 \leq \frac{1}{L}\left(f(x_*) - f(x_{N}) + \delta_1^{N-1} + \delta_2^{N-1}\right) +  \frac{q}{2}\norm{x_* - x_{N-1}}^2.
\end{align*}
Рекусрсивно получаем, что
\begin{align*}
\frac{1}{2}\norm{x_* - x_{N}}^2 \leq \sum_{i=1}^{N}\left(\frac{q^{N-i}}{L}(f(x_*) - f(x_{i})+ \delta_1^{i-1} + \delta_2^{i-1})\right) + \frac{q^{N}}{2}\norm{x_* - x_0}^2.
\end{align*}
Учитывая, что $\frac{1}{2}\norm{x_* - x_{N}}^2 \geq 0$ и определение $y_{N}$, мы получим:
\begin{align*}
 \frac{q^{N}}{2}\norm{x_* - x_0}^2
&\geq \sum_{i=1}^{N}\left(\frac{q^{N-i}}{L}(f(x_i) - f(x_*) - \delta_1^{i-1} - \delta_2^{i-1})\right)\\
&\geq (f(y_{N}) - f(x_*))\sum_{i=1}^{N}\frac{q^{N-i}}{L} - \frac{1}{L}\sum_{i=1}^{N}q^{N-i}(\delta_1^{i-1} + \delta_2^{i-1}).
\end{align*}
Разделим обе части последнего неравенства на $\sum_{i=1}^{N}\frac{q^{N-i}}{L}$:
\begin{align*}
f(y_{N}) - f(x_*)
&\leq \frac{\frac{q^{N}}{2}}{\sum_{i=1}^{N}\frac{q^{N-i}}{L}} \norm{x_* - x_0}^2 + \frac{1}{\sum_{i=1}^{N}q^{N-i}}\sum_{i=1}^{N}q^{N-i}(\delta_1^{i-1} + \delta_2^{i-1}).
\end{align*}        
Используя то, что $\sum_{i=1}^{N}\frac{q^{N-i}}{L} \geq \frac{1}{L}$ и $q^{N-i} \geq q^{N}$ для всех $i \geq 0$, мы получим неравенство:
\begin{align*}
f(y_{N}) - f(x_*) \leq \frac{L}{2}\min\left\{q^{N}, \frac{1}{N}\right\}\norm{x_* - x_0}^2 + \frac{1}{\sum_{i=1}^{N}q^{N-i}}\sum_{i=1}^{N}q^{N-i}(\delta_1^{i-1} + \delta_2^{i-1}).
\end{align*}
Данное неравенство и $q^N \leq \exp(-\frac{\mu}{L}N)$ завершают доказательство теоремы.
\end{proof}

Рассмотрим следствие теоремы~\ref{theorem:gm_model}.

\begin{theorem}
\label{theorem:gm_model_subgaussion}
    Пусть для последовательностей $\delta_1^k$ и $\delta_2^k$ ($k \geq 0$) верно предположение~\ref{assumption:delta_2}, тогда после $N$ шагов Алгоритма \ref{algorithm:GM} будет верно следующее неравенство:
    \begin{align}
    \E[f(y_N)] - f(x_*) \leq \min\left\{\frac{LR^2}{2N}, \frac{LR^2}{2}\exp\left(-\frac{\mu}{L}N\right)\right\} + \tilde{\delta}_1 + O(\hat{\delta}_2).
    \end{align}
    Если предположить, что $\mu = 0$, то с большой вероятностью
    \begin{equation*}
    f(y_N) - f(x_*) = \tilde{O}\left(\frac{LR^2}{N} + \frac{\hat{\delta}_1}{\sqrt{N}} + \tilde{\delta}_1 + \hat{\delta}_2\right).
    \end{equation*}
\end{theorem}

\begin{proof}
Первое неравенство можно получить используя стандартные неравенства для моментов субгуассовских случайных
величин \cite{juditsky2008}. Для второго неравенства надо заметить, что в случае $\mu = 0$ выполнено равенство:
\begin{align*}
\frac{1}{\sum_{i=1}^{N}q^{N-i}}\sum_{i=1}^{N}q^{N-i}(\delta_1^{i-1} + \delta_2^{i-1}) = \frac{1}{N}\sum_{i=1}^{N}(\delta_1^{i-1} + \delta_2^{i-1}).
\end{align*}
Для последнего слагаемого надо воспользоваться неравенствами концентрации для субгауссовских и субэкспоненциальных случайных величин \cite{juditsky2008}.
\end{proof}

\subsection{Быстрый градиентный метод для оптимизационных
задач, допускающих модель функции}
\label{section:fgm}

В случае быстрого градиентного метода нам понадобится изменить определение модели функции. Будем говорить, что функция $\psi_{\delta}(y,x)$ является $(\delta,L)$-моделью целевой функции $f(x)$, если для всех $x,y\in Q$ функция $\psi_{\delta}(y,x)$ -- выпукла по $y$, $\psi_{\delta}(x,x)\equiv 0$,
\begin{equation}\label{fgm_inexact_model}
f(x) + \psi_{\delta}(y, x) + \frac{\mu}{2}\|y-x\|^2_2 - \delta_1(y, x) \le f(y)  \le f(x) + \psi_{\delta}(y, x)+ \frac{L}{2}\|y-x\|^2_2 + \delta_2.
\end{equation}
Отметим, что теперь мы в общем случае предполагаем, что $\delta_1$ является функцией от двух аргументов $y, x \in Q$.

\begin{lemma}
\label{lemma:fast_strong_main}
	Для всех $x \in Q$ выполнено неравенство
	\begin{align*}
    		&A_{k+1} f(x_{k+1}) - A_{k} f(x_{k}) + \frac{{1 + A_{k+1} \mu}}{2}\norm{x - u_{k+1}}^2 - \frac{{1 + A_k \mu}}{2}\norm{x - u_k}^2\\ 
    		&\leq \alpha_{k+1}f(x) + A_{k}\delta_1^k(x_k,y_{k+1}) + \alpha_{k+1}\delta_1^k(x,y_{k+1}) + A_{k+1}\delta_2^k.
	\end{align*}
\end{lemma}
\begin{proof}
Воспользуемся \eqref{fgm_inexact_model}:
	\begin{align*}
	f(x_{k+1}) \leq f(y_{k+1}) + \psi_{\delta_k}(x_{k+1},y_{k+1})  + \frac{L}{2}\norm{x_{k+1} - y_{k+1}}^2 + \delta_2^k.
	\end{align*}
Из \eqref{x_def} и \eqref{y_def} для последовательностей $x_{k+1}$ и $y_{k+1}$ мы получим, что
	\begin{align*}
	f(x_{k+1})
	&\leq f(y_{k+1}) + \psi_{\delta_k}\left(\frac{\alpha_{k+1}u_{k+1} + A_k x_k}{A_{k+1}},y_{k+1}\right)\\
	&\hspace{2em}+ \frac{L}{2}\norm{\frac{\alpha_{k+1}u_{k+1} + A_k x_k}{A_{k+1}} - y_{k+1}}^2 + \delta_1^k\\
	&= f(y_{k+1}) + \psi_{\delta_k}\left(\frac{\alpha_{k+1}u_{k+1} + A_k x_k}{A_{k+1}},y_{k+1}\right)+\frac{L \alpha^2_{k+1}}{2 A^2_{k+1}}\norm{u_{k+1} - u_k}^2 + \delta_2^k.
	\end{align*}
Так как модель $\psi_{\delta_k}(\cdot,y_{k+1})$ выпуклая, то
	\begin{align*}
	f(x_{k+1})
	&\leq\frac{A_k}{A_{k+1}}\left(f(y_{k+1}) + \psi_{\delta_k}(x_k, y_{k+1})\right)+\frac{\alpha_{k+1}}{A_{k+1}}\left(f(y_{k+1}) + 
	 \psi_{\delta_k}(u_{k+1}, y_{k+1})\right)\\
	 &\hspace{2em}+ \frac{L \alpha^2_{k+1}}{2 A^2_{k+1}}\norm{u_{k+1} - u_k}^2 + \delta_2^k.
  \end{align*}
Из \eqref{alpha_def_strong} для последовательности $\alpha_{k+1}$ будет верно:
  \begin{align}
  \begin{split}
  \label{lemma:fast_str_conv:ineq_1}
	 f(x_{k+1})&\leq\frac{A_k}{A_{k+1}}\left(f(y_{k+1}) + \psi_{\delta_k}(x_k,y_{k+1})\right) +\frac{\alpha_{k+1}}{A_{k+1}}\Big(f(y_{k+1}) + \psi_{\delta_k}(u_{k+1},y_{k+1})\\
	 &\hspace{2em}+ \frac{{1 + A_k \mu }}{2 \alpha_{k+1}}\norm{u_{k+1} - u_k}^2\Big) + \delta_2^k.
  \end{split}
  \end{align}
Из леммы~\ref{lemma:str_conv} для оптимизационной задачи \eqref{fgm_step} будет следовать, что
\begin{align*}
&\alpha_{k+1}\psi_{\delta_k}(u_{k+1}, y_{k+1}) + \frac{1 + A_k \mu}{2}\norm{u_{k+1} - u_k}^2 +  \frac{\alpha_{k+1} \mu}{2}\norm{u_{k+1} - y_{k+1}}^2 \\
&\hspace{2em}+ \frac{1 + A_{k+1}\mu}{2}\norm{x - u_{k+1}}^2 \\ 
&\leq \alpha_{k+1}\psi_{\delta_k}(x, y_{k+1}) + \frac{1 + A_k \mu}{2}\norm{x - u_k}^2 + \frac{\alpha_{k+1} \mu}{2}\norm{x - y_{k+1}}^2.
\end{align*}
Так как $\frac{1}{2}\norm{u_{k+1} - y_{k+1}}^2 \geq 0$, то
\begin{align}
\begin{split}
  \label{lemma:fast_str_conv:ineq_2}
&\alpha_{k+1}\psi_{\delta_k}(u_{k+1}, y_{k+1}) + \frac{1 + A_k\mu}{2}\norm{u_{k+1} - u_k}^2\\ 
&\leq \alpha_{k+1}\psi_{\delta_k}(x, y_{k+1}) + \frac{1 + A_k\mu}{2}\norm{x - u_k}^2\\
&\hspace{2em}- \frac{1 + A_{k+1}\mu}{2}\norm{x - u_{k+1}}^2 + \frac{\alpha_{k+1} \mu}{2}\norm{x - y_{k+1}}^2.
\end{split}
\end{align}
Обьединим неравенства \eqref{lemma:fast_str_conv:ineq_1} и \eqref{lemma:fast_str_conv:ineq_2}, тогда
  \begin{align*}
  f(x_{k+1})
     &\leq \frac{A_k}{A_{k+1}} \left(f(y_{k+1}) + \psi_{\delta_k}(x_k,y_{k+1})\right)\\
     &\hspace{2em}+\frac{\alpha_{k+1}}{A_{k+1}}\Big(f(y_{k+1})
     + \psi_{\delta_k}(x,y_{k+1}) + { \frac{\mu}{2}\norm{x - y_{k+1}}^2}\\
	 &\hspace{2em}+ \frac{{1 + A_k \mu }}{2\alpha_{k+1}}\norm{x - u_k}^2 - \frac{{1 + A_{k+1} \mu }}{2\alpha_{k+1}}\norm{x - u_{k+1}}^2 \Big) + \delta_2^k.
\end{align*}
Воспользуемся левым неравенством из \eqref{fgm_inexact_model}:
\begin{align*}
    f(x_{k+1})
	 &\leq \frac{A_k}{A_{k+1}} f(x_k) + 
	 \frac{\alpha_{k+1}}{A_{k+1}} f(x) \\
	 &\hspace{2em}+ \frac{{1 + A_k \mu }}{2A_{k+1}}\norm{x - u_k}^2 - \frac{{1 + A_{k+1} \mu }}{2A_{k+1}}\norm{x - u_{k+1}}^2 \\
	 &\hspace{2em}+ \frac{A_{k}}{A_{k+1}}\delta_1^k(x_k,y_{k+1}) + \frac{\alpha_{k+1}}{A_{k+1}}\delta_1^k(x,y_{k+1}) + \delta_2^k.
	\end{align*}
Данное неравенство завершает доказательство леммы.
\end{proof}

\begin{theorem}
\label{theorem:fgm_model}
    После $N$ шагов Алгоритма \ref{algorithm:FGM} будет верно следующее неравенство:
    \begin{align*}
    f(x_{N}) - f(x_*) &\leq \frac{R^2}{2A_N}  + \frac{1}{A_N}\sum_{k=0}^{N-1}A_{k}\delta_1^k(x_k,y_{k+1}) \\
    &\hspace{2em}+ \frac{1}{A_N}\sum_{k=0}^{N-1}\alpha_{k+1}\delta_1^k(x_*,y_{k+1}) + \frac{1}{A_N}\sum_{k=0}^{N-1}A_{k+1}\delta_2^k.
    \end{align*}
\end{theorem}

\begin{proof}
Суммирая неравентсва из леммы~\ref{lemma:fast_strong_main} для $k$ от $0$ и $N-1$ и, взяв $x = x_*$, мы получим, что
\begin{align*}
		A_{N} f(x_{N})&\leq A_{N}f(x_*) + \frac{1}{2}\norm{x_* - u_0}^2 - \frac{1 + A_N\mu}{2}\norm{x_* - u_N}^2 + \sum_{k=0}^{N-1}A_{k}\delta_1^k(x_k,y_{k+1}) \\
    &\hspace{2em}+ \sum_{k=0}^{N-1}\alpha_{k+1}\delta_1^k(x_*,y_{k+1}) + \sum_{k=0}^{N-1}A_{k+1}\delta_2^k.
\end{align*}
Так как $\frac{1 + A_N\mu}{2}\norm{x_* - u_N}^2 \geq 0$, то
\begin{align*}
		A_{N} f(x_{N}) - A_{N}f(x_*) &\leq \frac{1}{2}\norm{x_* - u_0}^2 + \sum_{k=0}^{N-1}A_{k}\delta_1^k(x_k,y_{k+1}) \\
    &\hspace{2em}+ \sum_{k=0}^{N-1}\alpha_{k+1}\delta_1^k(x_*,y_{k+1}) + \sum_{k=0}^{N-1}A_{k+1}\delta_2^k.
\end{align*}
Последнее неравенство доказывает теорему.
\end{proof}


\begin{lemma}
\label{lemma:a_n_sequence}
Для всех $N \ge 1$, 
\begin{align*}
\frac{1}{A_N} \leq \min\left\{\frac{4L}{N^2}, 2L\exp\left(-\frac{N-1}{2}\sqrt{\frac{\mu }{L}}\right)\right\}.
\end{align*}
\end{lemma}

Результат леммы можно получить по аналогии с  \cite{devolder2013}, \cite{nesterov2013} (см. замечание 5.11.).

Далее нам будут полезны следующие предположения.  Похожее на предположение~\ref{assumption:delta_3} условие на детерминированный шум в градиенте можно встретить в работах \cite{stonyakin, stonyakin2020}.
\begin{assumption}
\label{assumption:delta_3}
Пусть даны две последовательности $\delta_1^k(x,y)$ и $\delta_2^k$ ($k \geq 0$). Случайная величина $\delta_1^k(x,y)$ имеет такое условное математическое ожидание, что
\begin{enumerate}
    \item $\E \left[\delta_1^k(x,y) |\delta_{1}^{k-1}(x,y),\delta_{2}^{k-1},\delta_{1}^{k-2}(x,y)\dots \right] \leq \tilde{\delta}_1^k(x - y)$ $ \forall x, y \in Q$, где $\tilde{\delta}_1^k(\cdot)$ есть неслучайная функция от одного аргумента.
    \item $\tilde{\delta}_1^k(\alpha z) \leq \alpha \tilde{\delta}_1^k(z)$ для всех $\alpha \geq 0$ и $z \in B(0, R)$.
    \item $\tilde{\delta}_1 < +\infty$, где $\tilde{\delta}_1 \geq \sup_{z \in B(0, R)} \tilde{\delta}_1^k(z)$.
\end{enumerate}
\end{assumption}

\begin{theorem}
\label{theorem:fgm_model_subgaussion}
    Пусть для последовательностей $\delta_1^k(x, y)$ и $\delta_2^k$ ($k \geq 0$) верно предположение~\ref{assumption:delta_2} и \ref{assumption:delta_3} для любых $x, y \in Q$, тогда после $N$ шагов Алгоритма \ref{algorithm:FGM} будет верно следующее неравенство:
    \begin{align*}
    \E[f(x_N)] - f(x_*) &\leq \min\left\{\frac{4LR^2}{N^2}, 2LR^2\exp\left(-\frac{N-1}{2}\sqrt{\frac{\mu }{L}}\right)\right\} + \tilde{\delta}_1 \\
    &\hspace{2em}+ O\left(\min\left\{N, \sqrt{\frac{L}{\mu}}\right\}\hat{\delta}_2\right).
    \end{align*}
    Предположим дополнительно, что $\mu = 0$, и для последовательности $\delta_1^k(x, y)$ выполнено предположение~\ref{assumption:delta_4}, тогда с большой вероятностью
    \begin{equation}
    \label{theorem:fgm_model_subgaussion:second}
    f(x_N) - f(x_*) = \tilde{O}\left(\frac{LR^2}{N^2} + \frac{\hat{\delta}_1}{\sqrt{N}} + \tilde{\delta}_1 + N\hat{\delta}_2\right).
    \end{equation}
\end{theorem}

\begin{proof}
Первое неравенство получается из тех же соображений, что и в доказательстве теоремы~\ref{theorem:gm_model_subgaussion} с учетом того, что $$\frac{1}{A_N}\sum_{k=0}^{N-1}A_{k+1} \leq O\left(\min\left\{N, \sqrt{\frac{L}{\mu}}\right\}\right)$$ (см. \cite{devolder2013}, лемма 5.11) и серии неравенств:
\begin{align*}
A_{k}\tilde{\delta}_1^k(x_k - y_{k+1}) = A_{k}\tilde{\delta}_1^k\left(\frac{\alpha_{k+1}}{A_k}(y_{k+1} - u_k)\right) \leq \alpha_{k+1} \tilde{\delta}_1^k(y_{k+1} - u_k) \leq \alpha_{k+1} \tilde{\delta}_1.
\end{align*}
В первом переходе мы воспользовались \eqref{y_def}. В предпоследнем и последнем переходе использовали предположение~\ref{assumption:delta_3}. Теперь докажем \eqref{theorem:fgm_model_subgaussion:second}.
Для доказательства неравенства 
$$\frac{1}{A_N}\sum_{k=0}^{N-1}A_{k+1}\delta_2^k \leq \tilde{O}(N\hat{\delta}_2)$$ нужно воспользоваться неравенством концентрации для субэкспоненциальных \\
случайных величин. Чтобы показать неравенство
$$\frac{1}{A_N}\sum_{k=0}^{N-1}\alpha_{k+1}\delta_1^k(x_*,y_{k+1}) \leq \tilde{O}\left(\frac{\hat{\delta}_1}{\sqrt{N}}\right)$$ нужно воспользоваться неравенством концентрации для субгауссовских случайных величин.
Неравенство $$\frac{1}{A_N}\sum_{k=0}^{N-1}A_{k}\delta_1^k(x_k,y_{k+1}) \leq \tilde{O}\left(\frac{\hat{\delta}_1}{\sqrt{N}}\right)$$ доказывается аналогично, как и предыдущее, но с учетом того, что 
\begin{align*}
A_{k}\hat{\delta}_1^k(x_k - y_{k+1}) = A_{k}\hat{\delta}_1^k\left(\frac{\alpha_{k+1}}{A_k}(y_{k+1} - u_k)\right) \leq \alpha_{k+1} \hat{\delta}_1^k(y_{k+1} - u_k) \leq \alpha_{k+1} \hat{\delta}_1.
\end{align*}
Во первом равенстве мы воспользовались \eqref{y_def}.
\end{proof}

\subsubsection{Максимум гладких функций} 
\label{subsection:maxmin}
Рассмотрим следующую задачу:

\begin{equation*}
F(x) := \max_{1 \leq i \leq m} f_i(x) \to \min_{x\in Q \subseteq \mathbb{R}^n}.
\end{equation*}

Будем предполагать, что $f_i(x)$ ($i \in [1, m]$) выпуклые и имеют $L$-Липшицев градиент, т.е. для всех $x,y\in Q$
\begin{equation}
\label{max_min:liptch}
f_i(x) + \la \nabla f_i(x), y - x \ra \leq f_i(y) \leq f_i(x) + \la \nabla f_i(x), y - x \ra + \frac{L}{2}\|y-x\|_2^2.
\end{equation}

Выберем в качестве модели функции $F(x)$:
\begin{equation*}
\psi_{\delta}(y, x) = \max_{1 \leq i \leq m} \left\{f_i(x) + \la \nabla f_i(x,\xi_i), y - x \ra \right\} - F(x),
\end{equation*}
где $\nabla f_i(x,\xi_i)$~--- независимые стохастические градиенты функций $f_i$, для которых выполнены условия \eqref{stoch_gradient_subgaussion}. Из \eqref{max_min:liptch} будет выполнено левое неравенство
\begin{align*}
F(y) &\geq F(x) + \psi_{\delta}(y, x) + \max_{1 \leq i \leq m} \left\{f_i(x) + \la \nabla f_i(x), y - x \ra \right\} \\
&\hspace{2em}- \max_{1 \leq i \leq m} \left\{f_i(x) + \la \nabla f_i(x,\xi_i), y - x \ra \right\}\\
&\geq F(x) + \psi_{\delta}(y, x) - \max_{1 \leq i \leq m} \left\{\la \nabla f_i(x,\xi_i) - \nabla f_i(x), y - x \ra \right\}
\end{align*}
и правое неравенство
\begin{align*}
F(y) &\leq F(x) + \psi_{\delta}(y, x) + \max_{1 \leq i \leq m} \left\{f_i(x) + \la \nabla f_i(x), y - x \ra \right\} \\
&\hspace{2em}- \max_{1 \leq i \leq m} \left\{f_i(x) + \la \nabla f_i(x,\xi_i), y - x \ra \right\} + \frac{L}{2}\norm{y - x}^2\\
&\leq F(x) + \psi_{\delta}(y, x) + \frac{1}{2L}\max_{1 \leq i \leq m}\norm{\nabla f_i(x) - \nabla f_i(x,\xi_i)}^2 + L \norm{y - x}^2.
\end{align*}
из \eqref{fgm_inexact_model}
с $$\delta_2 = \frac{1}{2L}\max_{1 \leq i \leq m}\norm{\nabla f_i(x) - \nabla f_i(x,\xi_i)}^2,$$
\begin{align*}
\delta_1(y,x) = \max_{1 \leq i \leq m} \left\{\la \nabla f_i(x,\xi_i) - \nabla f_i(x), y - x \ra \right\}
\end{align*} и $L := 2L$.
Из условий \eqref{stoch_gradient_subgaussion} получаем, что $\delta_1(y,x)$ является субгауссовской, а $\delta_2$~--- субэкспоненциальной случайной величиной, так как максимум субгауссовских (субэкспоненциальных) случайных величин есть субгауссовская (субэкспоненциальная) случайная величина. Более того, будут выполнены предположения~\ref{assumption:delta_4}, \ref{assumption:delta_2} и \ref{assumption:delta_3}.
Таким образом для текущей задачи с выбранной моделью применимы теоремы~\ref{theorem:gm_model_subgaussion} и \ref{theorem:fgm_model_subgaussion}.

\end{fulltext}


\end{document}